\pdfoutput=1 
\RequirePackage[l2tabu, orthodox]{nag} 
\documentclass[a4paper,reqno]{article} 
\usepackage[stretch=10]{microtype} 
\usepackage[german,french,UKenglish]{babel}
\usepackage[T1]{fontenc} 
\usepackage[utf8]{inputenc} 
\usepackage{csquotes} 
\usepackage[backend=biber,style=alphabetic,giveninits=true,isbn=false,sorting=nyt,maxnames=4,doi=false,url=false]{biblatex}
\usepackage{xpatch} 
\usepackage{letltxmacro} 
\usepackage{perpage} 
\usepackage{soul} 
\usepackage{xparse} 
\usepackage{amsmath}
\usepackage{amsfonts}
\usepackage{amssymb}
\usepackage{amsthm}
\usepackage{bm} 
\usepackage{mathrsfs}
\usepackage{mathdots}
\usepackage{enumitem}
\usepackage{calc}
\usepackage{tocloft} 
\usepackage{titlesec} 
\usepackage{textcomp} 
\usepackage{tikz}
\usepackage{tikz-cd} 
\usetikzlibrary{calc}
\usetikzlibrary{matrix,arrows}
\usetikzlibrary{decorations.markings}

\usepackage{verbatim} 
\usepackage{aliascnt} 
\usepackage[colorlinks=true,urlcolor=blue!70!black,citecolor=blue!60!black,linkcolor=blue!60!black,pdfauthor={R. van Dobben de Bruyn},pdftitle={The Hodge ring of varieties in positive characteristic}]{hyperref} 

\titlespacing\section{0pt}{12pt plus 4pt minus 2pt}{0pt plus 2pt minus 2pt}
\titlespacing\subsection{0pt}{12pt plus 4pt minus 2pt}{0pt plus 2pt minus 2pt}
\titlespacing\subsubsection{0pt}{0pt plus 2pt minus 2pt}{0pt plus 2pt minus 2pt}

\titleformat{\section}[block]{\Large\bfseries\scshape\filcenter}{\thesection.}{1ex}{}
\titleformat{\subsection}{\large\scshape\filcenter}{\thesubsection}{1ex}{}

\makeatletter

\ExplSyntaxOn
\cs_new:Npn \white_text:n #1
  {
    \llap{\textcolor{white}{\the\SOUL@syllable}\hspace{\fp_eval:n{0.01*#1} em}}
    \llap{\textcolor{white}{\the\SOUL@syllable}\hspace{-\fp_eval:n{0.01*#1} em}}
  }
\NewDocumentCommand{\whiten}{ m }
    {
      \int_step_function:nnnN {1}{1}{#1} \white_text:n
    }
\ExplSyntaxOff

\NewDocumentCommand{ \varul }{ D<>{5} O{0.2ex} O{0.1ex} +m } {%
\begingroup
\setul{#2}{#3}%
\def\SOUL@uleverysyllable{%
   \setbox0=\hbox{\the\SOUL@syllable}%
   \ifdim\dp0>\z@
      \SOUL@ulunderline{\phantom{\the\SOUL@syllable}}%
      \whiten{#1}%
      \llap{%
        \the\SOUL@syllable
        \SOUL@setkern\SOUL@charkern
      }%
   \else
       \SOUL@ulunderline{%
         \the\SOUL@syllable
         \SOUL@setkern\SOUL@charkern
       }%
   \fi}%
    \ul{#4}%
\endgroup
}

\makeatother

\makeatletter
\tikzcdset{
  open/.code     = {\tikzcdset{hook, circled};},
  closed/.code   = {\tikzcdset{hook, slashed};},
  open'/.code    = {\tikzcdset{hook', circled};},
  closed'/.code  = {\tikzcdset{hook', slashed};},
  circled/.code  = {\tikzcdset{markwith = {\draw (0,0) circle (.375ex);}};},
  slashed/.code  = {\tikzcdset{markwith = {\draw[-] (-.4ex,-.4ex) -- (.4ex,.4ex);}};},
  markwith/.code ={
    \pgfutil@ifundefined%
    {tikz@library@decorations.markings@loaded}%
    {\pgfutil@packageerror{tikz-cd}{You need to say %
      \string\usetikzlibrary{decorations.markings} to use arrows with markings}{}}{}%
    \pgfkeysalso{/tikz/postaction = {
      /tikz/decorate,
      /tikz/decoration={markings, mark = at position 0.5 with {#1}}}
    }
  },
}
\makeatother

\newtheoremstyle{thms}{1em}{0pt}{\itshape}{}{\bfseries}{.}{ }{}
\theoremstyle{thms}
\newtheorem{Thm}{Theorem}[section]				

\newaliascnt{Prop}{Thm}							
\newtheorem{Prop}[Prop]{Proposition}
\aliascntresetthe{Prop}

\newaliascnt{Lemma}{Thm}						
\newtheorem{Lemma}[Lemma]{Lemma}
\aliascntresetthe{Lemma}

\newaliascnt{Cor}{Thm}							
\newtheorem{Cor}[Cor]{Corollary}
\aliascntresetthe{Cor}

\newaliascnt{Conj}{Thm}							

\aliascntresetthe{Conj}

\newaliascnt{Question}{Thm}						

\aliascntresetthe{Question}

\newtheoremstyle{defs}{1em}{0pt}{}{}{\bfseries}{.}{ }{}
\theoremstyle{defs}
\newaliascnt{Rmk}{Thm}							
\newtheorem{Rmk}[Rmk]{Remark}
\aliascntresetthe{Rmk}

\newaliascnt{Fact}{Thm}							

\aliascntresetthe{Fact}

\newaliascnt{Def}{Thm}							
\newtheorem{Def}[Def]{Definition}
\aliascntresetthe{Def}

\newaliascnt{Ex}{Thm}							
\newtheorem{Ex}[Ex]{Example}
\aliascntresetthe{Ex}

\newaliascnt{Con}{Thm}							

\aliascntresetthe{Con}

\newaliascnt{Not}{Thm}							

\aliascntresetthe{Not}

\newaliascnt{Setup}{Thm}						

\aliascntresetthe{Setup}

\newaliascnt{Picture}{Thm}						

\aliascntresetthe{Picture}

\theoremstyle{thms}
\newtheorem{thm}{Theorem}
\newtheorem*{thm*}{Theorem}

\newtheorem*{lemma*}{Lemma}

\LetLtxMacro\oldproof\proof						
\renewcommand{\proof}[1][Proof]{\oldproof[#1]\unskip} 

\newenvironment{itemize*} 
  {\begin{itemize}
    \setlength{\itemsep}{1em}
    \setlength{\parskip}{-1em}
    \setlength{\topsep}{0pt}
    \setlength{\partopsep}{0pt}}
  {\end{itemize}}
  
\newenvironment{enumerate*}
  {\begin{enumerate}
    \setlength{\itemsep}{1em}
    \setlength{\parskip}{-1em}
    \setlength{\topsep}{0pt}
    \setlength{\partopsep}{0pt}}
  {\end{enumerate}}
  
\setlist{itemsep=0em,topsep=0cm,partopsep=0em,parsep=\lineskip}

\setlist[enumerate]{label=\normalfont(\arabic*)}

\setlist[itemize]{leftmargin=1.3em}

\newcommand{\PIC}{\mathbf{Pic}}
\newcommand{\Lie}{\operatorname{Lie}}
\newcommand{\HOM}{\mathbf{Hom}}

\newcommand{\Spec}{\operatorname{Spec}}
\newcommand{\Var}{\mathbf{Var}}
\newcommand{\Rig}{\mathbf{Rig}}
\newcommand{\Bl}[1]{\operatorname{Bl}_{#1}}
\newcommand{\im}{\operatorname{im}}
\newcommand{\rk}{\operatorname{rk}}
\newcommand{\pt}{\operatorname{pt}}

\newcommand{\fppf}{_{\operatorname{fppf}}}

\newcommand{\Fr}{\operatorname{Fr}}

\newcommand{\balpha}{\pmb\alpha}
\newcommand{\bmu}{\pmb\mu}

\newcommand{\ra}{\rightarrow}
\newcommand{\rA}{\longrightarrow}
\newcommand{\Ra}{\Rightarrow}

\newcommand{\LRa}{\Leftrightarrow}

\newcommand{\punct}[1]{\makebox[0pt][l]{\,#1}} 

\newcommand{\N}{\mathbf N}
\newcommand{\Z}{\mathbf Z}
\newcommand{\Q}{\mathbf Q}
\newcommand{\R}{\mathbf R}
\newcommand{\C}{\mathbf C}
\newcommand{\F}{\mathbf F}
\renewcommand{\P}{\mathbf P}
\newcommand{\A}{\mathbf A}
\newcommand{\G}{\mathbf G}

\newcommand{\hdR}{h_{\operatorname{dR}}}
\newcommand{\HdR}{H_{\operatorname{dR}}}
\newcommand{\HDR}{\mathcal{HDR}}
\newcommand{\DR}{\mathcal{DR}}
\newcommand{\dr}{\operatorname{dR}}
\renewcommand{\H}{\mathcal{H}}
\newcommand{\I}{\mathcal{I}}

\newcommand{\x}{^\times}

\DefineBibliographyStrings{UKenglish}{edition = {edition}} 
\AtEveryBibitem{\ifentrytype{book}{\clearfield{url}\clearfield{pages}}{}}
\renewbibmacro{in:}{}
\DeclareFieldFormat[book,incollection]{number}{\mkbibbold{#1}}
\DeclareFieldFormat[article]{volume}{\mkbibbold{#1}}
\DeclareFieldFormat*{title}{\mkbibitalic{#1}} 
\DeclareFieldFormat*{booktitle}{#1} 
\DeclareFieldFormat*{journaltitle}{#1}
\DeclareFieldFormat{postnote}{#1}
\DeclareFieldFormat{pages}{p. #1}
\DeclareFieldFormat[thesis]{type}{\bibstring{#1}\addcomma\addspace}
\DefineBibliographyStrings{english}{bibliography = {References},}
\AtEveryBibitem{\clearlist{publisher}\clearname{editor}\clearfield{location}}
\renewcommand*{\newunitpunct}{\addcomma\space}
\renewcommand*{\newblockpunct}{\addperiod\space}
\renewcommand*{\bibfont}{\footnotesize}

\renewcommand*{\mkbibnamefamily}[1]{\textsc{#1}}
\renewcommand*{\mkbibnameprefix}[1]{\textsc{#1}}

\newbibmacro{titlelink}[1]{%
	\iffieldundef{doi}{%
		\iffieldundef{url}{%
			#1
		}{%
			\href{\thefield{url}}{#1}%
		}%
	}{%
		\href{https://dx.doi.org/\thefield{doi}}{#1}%
	}%
}
\DeclareFieldFormat{title}{\usebibmacro{titlelink}{\mkbibemph{#1}}}

\xpatchbibdriver{article}
  {\newunit
   \usebibmacro{note+pages}}
  {\newunit}
  {}{}
\xpatchbibmacro{note+pages}{%
  \printfield{note}}
  {}
  {}{}
\renewbibmacro*{journal+issuetitle}{%
  \usebibmacro{journal}%
  \setunit*{\addspace}%
  \iffieldundef{series}
    {}
    {\newunit
     \printfield{series}%
     \setunit{\addspace}}%
  \usebibmacro{volume+number+eid}%
  \newunit
  \usebibmacro{note+pages}%
  \setunit{\addspace}%
  \usebibmacro{issue+date}%
  \setunit{\addcolon\space}%
  \usebibmacro{issue}%
  \newunit}
\xpatchbibdriver{inbook}
  {\newunit\newblock
  \usebibmacro{chapter+pages}}%
  {}%
  {}{}
\xpatchbibdriver{incollection}
  {\newunit\newblock
  \usebibmacro{chapter+pages}}%
  {}%
  {}{}
\xpatchbibdriver{inproceedings}
  {\newunit\newblock
  \usebibmacro{chapter+pages}}%
  {}%
  {}{}
\xapptobibmacro{maintitle+booktitle}
  {\printfield{pages}%
  \newunit}
  {}{}

\xpatchbibdriver{article}
  {\newunit\newblock
  \usebibmacro{in:}}%
  {\setunit{\addperiod\space}
  \printfield{note}
  \setunit{\addperiod\space}
  \usebibmacro{in:}}%
  {}{}
\xpatchbibmacro{addendum+pubstate}
  {\newunit\newblock
  \printfield{pubstate}}
  {\setunit{\addperiod\space}
  \printfield{pubstate}}
  {}{}
\xpatchbibdriver{misc}
  {\usebibmacro{doi+eprint+url}%
  \newunit\newblock
  \usebibmacro{addendum+pubstate}}
  {\usebibmacro{addendum+pubstate}%
  \newunit\newblock
  \usebibmacro{doi+eprint+url}}
  {}{}
\xpatchbibdriver{book}
  {\newunit\newblock
  \usebibmacro{series+number}%
  \newunit\newblock
  \printfield{note}}%
  {\setunit{\addperiod\space}%
  \usebibmacro{series+number}
  \setunit{\addperiod\space}%
  \printfield{note}}%
  {}{}
\xpatchbibdriver{inbook}
  {\newunit\newblock
  \usebibmacro{maintitle+booktitle}}%
  {\setunit{\addperiod\space}
  \usebibmacro{maintitle+booktitle}}%
  {}{}
\xpatchbibdriver{incollection}
  {\newunit\newblock
  \usebibmacro{in:}}%
  {\setunit{\addperiod\space}}
  {}{}
\xpatchbibdriver{inproceedings}
  {\newunit\newblock
  \usebibmacro{in:}}%
  {\setunit{\addperiod\space}}%
  {}{}
\xpatchbibdriver{incollection}
  {\newunit\newblock
  \usebibmacro{series+number}}%
  {\setunit{\addperiod\space}
  \usebibmacro{series+number}}%
  {}{}
\xpatchbibdriver{thesis}
  {\newunit\newblock
  \printfield{type}}%
  {\setunit{\addperiod\space}
  \printfield{type}}%
  {}{}
\xpatchbibdriver{misc}
  {\newunit\newblock
  \printfield{howpublished}}%
  {\setunit{\addperiod\space}\newblock
  \printfield{howpublished}}%
  {}{}
\xpatchbibmacro{doi+eprint+url}
  {\newunit\newblock
  \iftoggle{bbx:eprint}}%
  {\setunit{\addperiod\space}
  \iftoggle{bbx:eprint}}%
  {}{}
\xpatchbibmacro{series+number}
  {\printfield{series}}
  {\iffieldundef{shortseries}
    {\printfield{series}}
    {\printfield{shortseries}}}
  {}{}

\begin{document}

\renewcommand{\sectionautorefname}{Section}
\renewcommand{\subsectionautorefname}{Subsection}		

\begin{center}
\vspace*{-2em}
\noindent\makebox[\linewidth]{\rule{14cm}{0.4pt}}
\vspace{-.5em}

{\LARGE{\textsc{\textbf{The Hodge ring of varieties\\[.3em] in positive characteristic}}}}

\vspace{1.2em}

{\large{\textsc{Remy van Dobben de Bruyn}}}

\vspace{.25em}
\rule{8cm}{0.4pt}
\vspace{3em}
\end{center}

\renewcommand{\abstractname}{\small\bfseries\scshape Abstract}

\begin{abstract}\noindent
Let $k$ be a field of positive characteristic. We prove that the only linear relations between the Hodge numbers $h^{i,j}(X) = \dim H^j(X,\Omega_X^i)$ that hold for \emph{every} smooth proper variety $X$ over $k$ are the ones given by Serre duality. We also show that the only linear combinations of Hodge numbers that are birational invariants of $X$ are given by the span of the $h^{i,0}(X)$ and the $h^{0,j}(X)$ (and their duals $h^{i,n}(X)$ and $h^{n,j}(X)$). The corresponding statements for compact K\"ahler manifolds were proven by Kotschick and Schreieder \cite{KS}. 
\end{abstract}


\phantomsection
\section*{Introduction}
The Hodge numbers $h^{i,j}(X) = \dim H^j(X,\Omega^i_X)$ of a compact K\"ahler manifold~$X$ satisfy the relations $h^{i,j} = h^{n-i,n-j}$ (Serre duality) and $h^{i,j} = h^{j,i}$ (Hodge symmetry). Kotschick and Schreieder showed \cite[Thm.~1,~consequence~(2)]{KS} that these are the only \emph{universal} linear relations between the $h^{i,j}$, i.e.\ the only ones that are satisfied for every compact K\"ahler manifold $X$ of dimension $n$.

Serre constructed smooth projective varieties in characteristic $p > 0$ for which Hodge symmetry fails \cite[Prop.\ 16]{SerMex}; however Serre duality still holds. Thus, the natural analogue of the result of Kotschick--Schreieder is the following.

\begin{thm}\label{Thm Hodge}
The only universal linear relations between the Hodge numbers $h^{i,j}(X)$ of smooth proper varieties $X$ over an arbitrary field $k$ of characteristic $p > 0$ are the ones given by Serre duality.
\end{thm}

This is proven in \autoref{Thm linear relations Hodge-de Rham}. The strategy is identical to that of \cite{KS}: compute the subring of $\Z[x,y,z]$ generated by the formal Hodge polynomials
\[
h(X) = \left(\sum_{i,j} h^{i,j}(X)x^i y^j\right)z^{\dim X}.
\]
The term $z^{\dim X}$ is new in \cite{KS}, and its introduction ensures that varieties of different dimensions cannot be identified with each other. It defines a grading on $\Z[x,y,z]$ by dimension, i.e.\ $\Z[x,y,z]_n = \Z[x,y]z^n$.

In characteristic $0$, Kotschick and Schreieder prove \cite[Cor.\ 3]{KS} that the graded subring of $\Z[x,y,z]$ given in degree $n$ by the equations $h^{i,j} = h^{n-i,n-j}$ and $h^{i,j} = h^{j,i}$ is the free algebra $\Z[h(\P^1),h(E),h(\P^2)]$ on the Hodge polynomials of $\P^1$, an elliptic curve $E$, and~$\P^2$.

We prove the natural analogue in characteristic $p$:

\begin{thm}\label{Thm Hodge ring}
Let $\H \subseteq \Z[x,y,z]$ be the graded subring defined by
\[
\H_n = \left\{\left(\sum_{i,j=0}^n h^{i,j}x^iy^j\right)z^n\ \Bigg|\ h^{i,j} = h^{n-i,n-j} \text{ for all } i,j\right\}.
\]
Then $\H$ is generated by $h(\P^1)$, $h(E)$, $h(\P^2)$, and $h(S)$, where $E$ is any elliptic curve and $S$ is any surface with $h^{1,0}(S) - h^{0,1}(S) = \pm 1$. Moreover, $h(S)$ satisfies a monic quadratic relation over the free polynomial algebra $\Z[h(\P^1),h(E),h(\P^2)]$.
\end{thm}

This is proven in \autoref{Cor generated by P^1, E, P^2, and S}. An example of a surface $S$ satisfying the hypothesis is given by Serre's counterexample to Hodge symmetry \cite[Prop.~16]{SerMex}, since $h^{0,1}(S) = 0$ and $h^{1,0}(S) = 1$. We recall the construction in \autoref{Prop example Hodge symmetry}.

There are ``twice as many'' possible Hodge polynomials in characteristic $p$ as in characteristic $0$: the Hodge diamonds now only have a $\Z/2$-symmetry instead of a $(\Z/2)^2$-symmetry. Thus, a Hilbert polynomial calculation suggests that the Hodge ring of varieties in characteristic $p > 0$ should be a quadratic extension of the one in characteristic $0$. This is indeed confirmed by \autoref{Thm Hodge ring}.

Similar to \cite[Thm.\ 2]{KS}, we also get a statement on birational invariants:

\begin{thm}\label{Thm Hodge birational}
A linear combination of the Hodge numbers $h^{i,j}(X)$ is a birational invariant of the variety $X$ (over a field $k$ of characteristic $p$) if and only if it is in the span of the $h^{0,j} = h^{n,n-j}$ and the $h^{i,0} = h^{n-i,n}$.
\end{thm}

This is proven in \autoref{Thm linear combinations birational invariant} below. The result in characteristic $0$ is the same \cite[Thm.\ 2]{KS}, except that there $h^{i,0}$ is actually equal to $h^{0,i}$ by Hodge symmetry. The fact that the $h^{0,j}$ are birational invariants in characteristic $p > 0$ is due to Chatzistamatiou and R\"ulling \makebox{\cite[Thm.\ 1]{ChaRul};} for $h^{i,0}$ this is classical.

Finally we address the failure of degeneration of the Hodge--de Rham spectral sequence:

\begin{thm}\label{Thm Hodge de Rham}
The only linear relations between the Hodge numbers $h^{i,j}(X)$ and the de Rham numbers $\hdR^i(X)$ of smooth proper varieties $X$ over $k$ are the relations spanned by
\begin{itemize}
\item Serre duality: $h^{i,j}(X) = h^{n-i,n-j}(X)$;
\item Poincar\'e duality: $\hdR^i(X) = \hdR^{2n-i}(X)$;
\item Components: $h^{0,0}(X) = \hdR^0(X)$;
\item Euler characteristic: $\sum_{i,j} (-1)^{i+j} h^{i,j}(X) = \sum_i (-1)^i \hdR^i(X)$.
\end{itemize}
\end{thm}

This is proven in \autoref{Thm linear relations Hodge-de Rham}.

\subsection*{Outline of the paper}
In \autoref{Sec Hodge symmetry} we recall Serre's example of the failure of Hodge symmetry, and make slight improvements that we will use later. In \autoref{Sec Hodge-de Rham degeneration} we recall an example of W.\ Lang of a similar flavour of a surface for which the Hodge--de Rham spectral sequence does not degenerate.

In \autoref{Sec Hodge ring} we define the \emph{Hodge ring of varieties} as the graded subring of $\Z[x,y,z]$ where Serre duality holds, and show that it is generated by the Hodge polynomials of $\P^1$, any elliptic curve $E$, $\P^2$, and the surface $S$ of \autoref{Sec Hodge symmetry}. This proves \autoref{Thm Hodge}. In \autoref{Sec birational} we use this presentation to study birational invariants, proving \autoref{Thm Hodge birational}.

In \autoref{Sec de Rham} we define the \emph{de Rham ring of varieties} as the graded subring of $\Z[t,z]$ where Poincar\'e duality holds, and produce a similar presentation in terms of $\P^1$, $E$, $\P^2$, and $S$. Finally, in \autoref{Sec Hodge-de Rham} we combine the information of the previous sections into a \emph{Hodge--de Rham ring of varieties}, and we use the surface $S$ of \autoref{Sec Hodge-de Rham degeneration} to show that the Hodge--de Rham ring can be generated by varieties. This gives \autoref{Thm Hodge de Rham}. 

\subsection*{Notation}
Throughout, $p$ will be a prime number, and $k$ will be a field of characteristic $p$ (fixed once and for all). A \emph{variety} over a field $k$ will be a geometrically integral, finite type, separated $k$-scheme. We write $\Var_k$ for the set of isomorphism classes of smooth proper $k$-varieties.

For a variety $X$ over a field $k$, we will write $H^{i,j}(X) = H^j(X,\Omega_X^i)$, and denote its dimension by $h^{i,j}(X) = h^j(X,\Omega_X^i)$ (we avoid the usual $h^{p,q}(X)$ as it leads to a clash of notation). Similarly, the algebraic de Rham cohomology is denoted $\HdR^i(X)$, and its dimension is $\hdR^i(X)$. We warn the reader that $h^{i,j}(X)$ and $h^{j,i}(X)$ may differ, and the sum $\sum_{i+j=m} h^{i,j}(X)$ may not equal $\hdR^m(X)$.

The Picard functor $\PIC_X = R^1\pi_* \mathbf G_m$ of a smooth proper variety $\pi \colon X \to \Spec k$ is representable by a scheme \cite[Exp.~XII,~Cor.~1.5(a)]{SGA6}. We denote its identity component by $\PIC_X^0$, and the union of the torsion components by $\PIC^\tau_X$. Note that neither is reduced in general.

\subsection*{Acknowledgements}
{\small
I am grateful to Stefan Schreieder for his encouragement to work on the positive characteristic analogues of \cite{KS}, as well as to Johan de Jong for his guidance in the early stages of this project. I benefited from conversations with Bhargav Bhatt, Raymond Cheng, Johan Commelin, H\'el\`ene Esnault, Raju Krishnamoorthy, Shizhang Li, Daniel Litt, Milan Lopuha\"a-Zwakenberg, Qixiao Ma, Xuanyu Pan, and Burt Totaro. I am particularly grateful to Xuanyu Pan for \autoref{Rmk not smooth}, and to Qixiao Ma for providing the idea for \autoref{Ex S and T}.
Part of this work was carried out with support of the Oswald Veblen Fund at the Institute for Advanced Study.
}

\numberwithin{equation}{section}

\section{Failure of Hodge symmetry}\label{Sec Hodge symmetry}
Following Serre \cite[Prop.\ 16]{SerMex}, we construct a smooth projective surface $X$ over $\F_p$ with $h^0(X,\Omega_X^1) = 0$ but $h^1(X,\mathcal O_X) = 1$. We do a little better than Serre's example: our $X$ is defined over the prime field $\F_p$, admits a lift to $\Z_p$, and we include the exact computation of $h^1(X,\mathcal O_X)$.

We start with a well-known lemma that is useful for the constructions in this section and the next.

\begin{Lemma}\label{Lem G-torsor sequence}
Let $X$ and $Y$ be proper $k$-varieties, let $G$ be a finite group scheme over $k$, and let $f \colon Y \to X$ be a $G$-torsor. Then there is a short exact sequence
\[
0 \to G^\vee \to \PIC_X \to (\PIC_Y)^G
\]
on the big flat site $(\Spec k)\fppf$, where $G^\vee$ is the Cartier dual of $G$.
\end{Lemma}

\begin{proof}
We have a commutative diagram
\begin{equation*}
\begin{tikzcd}
Y \ar{r}{\pi_Y}\ar{d}[swap]{f} & \Spec k\ar{d}[swap,yshift=.15em]{f^{\text{univ}}}\ar[start anchor=south east, end anchor=north west, equal]{rd} & \\
X \ar{r}[swap]{g} & BG \ar{r}[swap]{\pi_{BG}} & \Spec k
\end{tikzcd}
\end{equation*}
whose left hand square is a pullback, where $\pi_{\mathscr X} \colon \mathscr X \to \Spec k$ denotes the structure map of an algebraic stack $\mathscr X$ over $k$. The Grothendieck spectral sequence for
\begin{equation*}
\begin{tikzcd}
X\fppf \ar{r}{g_*} & BG\fppf \ar{r}{\pi_{BG,*}} & (\Spec k)\fppf
\end{tikzcd}
\end{equation*}
gives a short exact sequence of low degree terms
\begin{equation}
0 \to R^1\pi_{BG,*}(g_*\G_m) \to R^1\pi_{X,*}\G_m \to \pi_{BG,*}(R^1g_*\G_m).\label{Dia low degree terms}
\end{equation}
The middle term is $\PIC_X$, and the pullback of $R^1g_*\G_m$ along $f^{\text{univ}}$ is $\PIC_Y$. Since $\pi_{BG,*} \colon BG\fppf \to (\Spec k)\fppf$ is computed by pulling back along $f^{\text{univ}}$ and taking $G$-invariants, the third term of (\ref{Dia low degree terms}) is $(\PIC_Y)^G$.

Since $\pi_Y$ is proper with geometrically connected fibres, the same goes for $g$. This forces $g_* \G_m = \G_m$, so the first term of (\ref{Dia low degree terms}) is $R^1\pi_{BG,*}\G_m = \PIC_{BG}$. A line bundle on $BG$ is trivialised on $f^{\text{univ}} \colon \Spec k \to BG$, so it gives a cocycle on $\Spec k \times_{BG} \Spec k \cong G$. This is a function $\phi \colon G \to \G_m$, and the cocycle condition amounts to the condition that $\phi$ is a homomorphism. We finally conclude that $\PIC_{BG} \cong \HOM(G,\G_m) = G^\vee$.
\end{proof}
\vskip-\lastskip
There are more elementary proofs of \autoref{Lem G-torsor sequence} by trivialising a line bundle $\mathscr L$ on $X$ by pulling back to $Y$, and using $H^0(Y,\mathcal O_Y) = k$ to construct a morphism $G \to \G_m$. The advantage of our proof above is that it gives the obstruction to surjectivity of $\PIC_X \to (\PIC_Y)^G$: the next term in the sequence is $R^2\pi_{BG,*}\G_m$. Jensen showed \cite{Jen} that this group vanishes for most of the groups we're interested in, e.g.~$\Z/p$, $\balpha_p$, and $\bmu_p$, but we don't need this.

\begin{Cor}\label{Cor quotient of hypersurface}
Let $X$ be a proper $k$-variety, $G$ a finite group scheme over $k$, and $Y \to X$ a $G$-torsor. If $\PIC_Y^\tau = 0$, then $\PIC_X^\tau \cong G^\vee$.
\end{Cor}

\begin{proof}
The image of $\PIC_X^\tau$ always lands in $\PIC_Y^\tau$, which is $0$ by assumption. The restriction of the short exact sequence of \autoref{Lem G-torsor sequence} to the respective torsion components gives the result.
\end{proof}
\vskip-\lastskip
This allows us to construct a surface for which Hodge symmetry fails, following Serre \cite{SerMex}.

\begin{Prop}\label{Prop example Hodge symmetry}
There exists a smooth projective surface $X$ over $\F_p$ admitting a lift to $\Z_p$ such that $H^0(X,\Omega_X^1) = 0$ but $h^1(X,\mathcal O_X) = 1$.
\end{Prop}

\begin{proof}
Let $G \to \Spec \Z_p$ be the constant \'etale group scheme $\Z/p$. Choose a linear action of $G$ on $\P^N_{\Z_p}$ for some $N \gg 0$ such that the fixed point locus of the special fibre has codimension at least $3$. For example, we can take $3$ copies of the regular representation, and then projectivise. Form the quotient $\mathcal Z = \P^N/G$, which is smooth away from the image of the fixed locus.

By the codimension $\geq 3$ assumption, repeatedly applying Poonen's Bertini theorem \cite{Poo} produces a projective surface $\mathcal X \subseteq \mathcal Z^{\text{sm}}$ over $\Z_p$ whose special fibre is smooth, hence $\mathcal X$ is smooth over $\Z_p$. Its inverse image $\mathcal Y$ in $\P^N_{\Z_p}$ is a complete intersection, and $\mathcal Y \to \mathcal X$ is a $G$-torsor.

The special fibre $Y$ of $\mathcal Y$ is smooth since $X$ and $G$ are, so $H^0(Y,\Omega_Y^1) = 0$, $H^1(Y,\mathcal O_Y) = 0$ and $\PIC_Y^\tau = 0$ \cite[Exp.\ XI,\ Thm.\ 1.8]{SGA7II}. Then \autoref{Cor quotient of hypersurface} gives $\PIC_X^\tau = (\Z/p)^\vee = \bmu_p$, so
\[
H^1(X,\mathcal O_X) = \Lie(\PIC_X^\tau) = \Lie(\bmu_p)
\]
has dimension $1$. On the other hand, $Y \to X$ is \'etale since it is a $(\Z/p)$-torsor, so $H^0(X,\Omega_X^1) \to H^0(Y,\Omega_Y^1)$ is injective, forcing $H^0(X,\Omega_X^1) = 0$.
\end{proof}

\begin{Rmk}\label{Rmk Hodge-de Rham lift}
Because the surface $X$ constructed in \autoref{Prop example Hodge symmetry} lifts to $\Z_p$, the Hodge--de Rham spectral sequence degenerates \cite{DelIll}. For $H^0(X,\Omega_X^1)$ and $H^1(X,\mathcal O_X)$, this can also be checked by hand using cocycles (analogous to the proof of \autoref{Prop example Hodge-de Rham degeneration} below).
\end{Rmk}

\section{Failure of Hodge--de Rham degeneration}\label{Sec Hodge-de Rham degeneration}
Inspired by a construction of W.\ Lang \cite{Lang}, we do a precise computation of $H^0(X,\Omega_X^1)$, $H^1(X,\mathcal O_X)$, and $\HdR^1(X)$ when $X$ is an $\balpha_p$-quotient of a (singular) complete intersection. Lang moreover shows that $X$ lifts to a degree $2$ ramified extension of $\Z_p$, but this plays no role for us. 

A lot is written about structure theory and singularities of $\balpha_p$-torsors and more generally purely inseparable quotients \cite{Eke, KimNiit, AraAvra,DemGab,Tzio}, but we only need elementary results.

\begin{Rmk}\label{Rmk alpha_p torsor}
If $A$ is a ring of characteristic $p$, then an $\balpha_p$-torsor is given by $A \to A[z]/(z^p - f)$ for some $f \in A$, and it is the trivial torsor if and only if $f$ is a $p$-th power. Indeed, this follows from the exact sequence
\[
A \stackrel {(-)^p}\rA A \to H^1(\Spec A, \balpha_p) \to 0.
\]
For a general $\balpha_p$-torsor $\pi \colon Y \to X$, there are affine opens $\Spec A_i = U_i \subseteq X$ and sections $z_i \in \mathcal O_Y(\pi^{-1}(U_i))$ and $f_i \in \mathcal O_X(U_i)$ such that $\mathcal O_X(U_i) \to \mathcal O_Y(\pi^{-1}(U_i))$ is given by $A_i \to A_i[z_i]/(z_i^p - f_i)$. We have $f_i - f_j = g_{ij}^p$ for some $g_{ij} \in \mathcal O_X(U_{ij})$, hence the $1$-forms $\omega_i = df_i$ glue to a global $1$-form $\omega$.
\end{Rmk}

\begin{Lemma}\label{Lem torsor}
Let $X$ be a normal integral scheme of characteristic $p$, and let $\pi \colon Y \to X$ be an $\balpha_p$-torsor, with $f_i$, $z_i$, $g_{ij}$, $\omega_i$, and $\omega$ as in \autoref{Rmk alpha_p torsor}. Then the following are equivalent:
\begin{enumerate}
\item $\pi$ is not the trivial torsor;\label{Item nontrivial torsor}
\item no $f_i$ is a $p$-th power;\label{Item not a p-th power}
\item $Y$ is integral.\label{Item integral}
\item $\omega$ is not zero in $H^0(X,\Omega_{X/\F_p}^1)$ (or even in $\Omega_{K(X)/\F_p}^1$);\label{Item df nonzero}
\end{enumerate}
\end{Lemma}

\begin{proof}
If one $f_i$ is a $p$-th power, then they all are, since they differ by $p$-th powers. Hence, $\pi$ is trivial in this case. This proves $\ref{Item nontrivial torsor} \Ra \ref{Item not a p-th power}$; the converse is trivial. Since $X$ is normal, $f_i$ is a $p$-th power in $A_i$ if and only if it is so in $K = K(X)$. Thus, if no $f_i$ is a $p$-th power, then all the rings $K[z_i]/(z_i^p-f_i)$ are fields, so $A_i[z_i]/(z_i^p-f_i)$ is a domain. This proves $\ref{Item not a p-th power} \Ra \ref{Item integral}$, and again the converse is trivial. Finally, $\ref{Item not a p-th power} \LRa \ref{Item df nonzero}$ is \cite[Tag \href{https://stacks.math.columbia.edu/tag/07P2}{07P2}]{Stacks}.
\end{proof}

\begin{Lemma}\label{Lem torsor smooth}
Let $k$ be a field of characteristic $p$, let $X$ and $Y$ be finite type $k$-schemes, and let $\pi \colon Y \to X$ be a morphism of $k$-schemes that is an $\balpha_p$-torsor. If $Y$ is smooth, then so is $X$.
\end{Lemma}

\begin{proof}
We may replace $k$ by $\bar k$ and `smooth' by `regular'. Since $Y$ is regular, Kunz's theorem \cite[Thm.\ 2.1]{Kunz} implies that $\Fr_Y \colon Y \to Y$ is flat. We have a factorisation
\vspace{-.75em}
\begin{equation*}
\begin{tikzcd}
Y \ar{r}{\pi}\ar[bend left]{rr}{Fr_Y} & X \ar{r}\ar[bend right]{rr}[swap]{\Fr_X} & Y \ar{r}{\pi} & X\punct{,}
\end{tikzcd}
\end{equation*}
where $\pi$ and $\Fr_Y$ are flat. We conclude that $\Fr_X \circ \pi = \pi \circ \Fr_Y$ is flat, hence $\Fr_X$ is flat since $\pi$ is faithfully flat. Then Kunz's theorem \cite[Thm.\ 2.1]{Kunz} implies that $X$ is regular.
\end{proof}

\begin{Prop}\label{Prop example Hodge-de Rham degeneration}
There exists a smooth projective surface $X$ over $\F_p$ for which $H^0(X,\Omega_X^1)$, $H^1(X,\mathcal O_X)$, and $\HdR^1(X)$ are all $1$-dimensional. In particular, the Hodge--de Rham spectral sequence of $X$ does not degenerate.
\end{Prop}

\begin{proof}
Let $G \to \Spec \F_p$ be the group scheme $\balpha_p$, and choose a linear action of $G$ on $\P^N_{\F_p}$ for some $N \gg 0$ such that the fixed point locus has codimension at least $3$. For example, we take 3 copies of the regular representation, and then projectivise \cite[Lem.\ 4.2.2]{Ray}. Form the quotient $Z = \P^N/G$, which is an $\balpha_p$-torsor away from the image $Z^{\text{fix}} \subseteq Z$ of the fixed locus \cite[Tag \href{https://stacks.math.columbia.edu/tag/07S7}{07S7}]{Stacks}. In particular, $Z$ is smooth outside $Z^{\text{fix}}$ by \autoref{Lem torsor smooth}. Since $\P^N$ is geometrically integral, the differential form $\omega$ of \autoref{Rmk alpha_p torsor} is nontrivial by \autoref{Lem torsor}.

Repeatedly applying Poonen's Bertini theorem \cite{Poo} produces a smooth projective surface $X \subseteq Z \setminus Z^{\text{fix}}$, which we may choose such that $\omega|_X$ is not identically zero (for example by specifying a tangency condition at a point $x \in Z \setminus Z^{\text{fix}}$).

The inverse image $Y \subseteq \P^N$ of $X$ is a complete intersection, and $\pi \colon Y \to X$ is an $\balpha_p$-torsor since $Y \subseteq Z\setminus Z^{\text{fix}}$. Hence $Y$ is geometrically integral by \autoref{Lem torsor}, since $\omega|_{X_{\bar \F_p}} \neq 0$. Since $Y$ is a geometrically integral complete intersection of dimension $\geq 2$, we find $H^1(Y,\mathcal O_Y) = 0$ and $H^0(Y,\Omega_Y^1) = 0$.
In particular, $\PIC_Y^0 = 0$, so \autoref{Lem G-torsor sequence} gives
\[
\PIC_X^0 = \balpha_p^\vee = \balpha_p.
\]
Thus, $H^1(X,\mathcal O_X) = \Lie(\PIC_X^0)$ is $1$-dimensional. The short exact sequence
\[
0 \to H^1(X,\balpha_p) \to H^1(X,\mathcal O_X) \stackrel{(-)^p}\rA H^1(X,\mathcal O_X)
\]
shows that $H^1(X,\mathcal O_X)$ is spanned by the image of the nontrivial $\balpha_p$-torsor $\pi$. In the notation of \autoref{Rmk alpha_p torsor}, the generator of $H^1(X,\mathcal O_X)$ is given by the \u Cech $1$-cocycle $(g_{ij})$. We claim that its image in $H^1(X,\Omega_X^1)$ is nonzero, i.e.\ $(dg_{ij})$ is not a \u Cech coboundary. Suppose it were; say $(dg_{ij}) = \eta_i - \eta_j$ for forms $\eta_i \in H^0(U_i,\Omega_{U_i}^1)$. Then consider the $1$-form $dz_i - \eta_i$ on $\pi^{-1}(U_i)$. We have
\[
(z_i-z_j)^p = f_i - f_j = g_{ij}^p,
\]
hence $z_i - z_j = g_{ij}$ since $Y$ is integral. Hence,
\[
(dz_i - \eta_i)- (dz_j - \eta_j) = dg_{ij} - \eta_i + \eta_j = 0,
\]
showing that the $dz_i - \eta_i$ glue to a global $1$-form on $Y$, which is nonzero because $dz_i$ is not pulled back from $U_i$. This contradicts the vanishing of $H^0(Y,\Omega_Y^1)$. We conclude that the generator $(g_{ij})$ of $H^1(X,\mathcal O_X)$ does not survive the Hodge--de Rham spectral sequence.

On the other hand, the kernel of $\Omega_X^1 \to \pi_* \Omega_Y^1$ is locally generated by $\omega_i$, since the cotangent complex of $\pi^{-1}(U_i) \to U_i$ is the complex
\[
\mathcal I_i/\mathcal I_i^2 \stackrel 0\rA \Omega_{\pi^{-1}(U_i)/U_i}^1,
\]
where $\mathcal I_i$ is the ideal sheaf on $U_i \times \A^1$ given by $z_i^p-f_i$. We get an exact sequence
\[
0 \to \omega\mathcal O_X \to \Omega_X^1 \to \pi_*\Omega_Y^1,
\]
so vanishing of $H^0(Y,\Omega_Y^1)$ gives $H^0(X,\omega\mathcal O_X) = H^0(X,\Omega_X^1)$, i.e.\ $H^0(X,\Omega_X^1)$ is $1$-dimensional and generated by $\omega$. Note that $d\omega = 0$ since $\omega$ is locally exact, so $\omega$ survives the Hodge--de Rham spectral sequence. Hence, $\hdR^1(X) = 1$.
\end{proof}
\vskip-\lastskip
The constructions of \autoref{Prop example Hodge symmetry} and \autoref{Prop example Hodge-de Rham degeneration} are both covered by \cite[Prop.\ 4.2.3]{Ray}, but we needed a more detailed analysis of the Hodge numbers. Using Poonen's Bertini theorems, we were also able to construct examples over the prime field $\F_p$. The smoothness claim for $\balpha_p$-quotients is not explained in [\emph{loc.\ cit.}].

\begin{Rmk}\label{Rmk not smooth}
In general we cannot expect the complete intersection $Y \subseteq \P^N$ of the proof of \autoref{Prop example Hodge-de Rham degeneration} to be smooth. Indeed, the degree of $Y$ is large with respect to $p$ in order for $Y$ to be $\balpha_p$-invariant. But a smooth complete intersection of sufficiently high degree does not admit any infinitesimal automorphisms, so in particular cannot have a free $\balpha_p$-action.
\end{Rmk}

\section{The Hodge ring of varieties}\label{Sec Hodge ring}
We modify the definition of the Hodge ring of \cite{KS} to account for the failure of Hodge symmetry in characteristic $p > 0$.

\begin{Def}\label{Def Hodge ring}
Consider $\Z[x,y,z]$ as a graded ring where $x$ and $y$ both have degree $0$ and $z$ has degree $1$. The \emph{Hodge ring of varieties over $k$} is the graded subring $\H_* \subseteq \Z[x,y,z]$ whose $n$-th graded piece is
\[
\H_n = \left\{\left(\ \sum_{i,j = 0}^nh^{i,j}x^iy^j\right)z^n\ \Bigg|\ \begin{array}{c}h^{n-i,n-j} = h^{i,j} \\ \text{ for all } i,j \in \{0,\ldots,n\}\end{array}\right\}.
\]
This is different from the notation of \cite{KS}, where $\H_*$ denotes the ring where moreover Hodge symmetry holds.
\end{Def}

\begin{Def}\label{Def h(X)}
Write $h \colon \Var_k \to \H_*$ for the map that sends an $n$-dimensional variety $X$ to its \emph{(formal) Hodge polynomial}
\[
h(X) = \left(\ \sum_{i,j=0}^nh^{i,j}(X)x^iy^j\right)z^n.
\]
where $h^{i,j}(X) = h^j(X,\Omega_X^i)$ as usual. The K\"unneth formula \cite[Tag \href{https://stacks.math.columbia.edu/tag/0BED}{0BED}]{Stacks} shows that $h$ preserves products, i.e.\ $h(X \times Y) = h(X) \cdot h(Y)$ for all $X, Y \in \Var_k$.

To justify the name \emph{Hodge ring of varieties}, we show in \autoref{Cor Hodge generated by varieties} that $\H_*$ is the subring (equivalently, subgroup) of $\Z[x,y,z]$ generated by the image of $h$.
\end{Def}

\begin{Thm}\label{Thm Hodge ring presentation}
Consider the graded ring $\Z[A,B,C,D]$ where $A$ and $B$ have degree $1$, and $C$ and $D$ have degree $2$. Then the map
\[
\phi \colon \Z[A,B,C,D] \to \H_*
\]
given by
\begin{align*}
\phi(A) &= (1+xy)z, & \phi(C) &= xy \cdot z^2,\\
\phi(B) &= (x+y)z,  & \phi(D) &= (x+xy^2)z^2
\end{align*}
is a surjection of graded rings, with kernel generated by
\[
G := D^2 - ABD + C(A^2+B^2-4C).
\]
\end{Thm}

\begin{Rmk}
Kotschick and Schreieder show \cite[Thm.\ 6]{KS} that $\Z[A,B,C]$ maps isomorphically onto the subring of $H_*$ where Hodge symmetry holds. Our proof is a modification and simplification of theirs.
\end{Rmk}

\begin{Lemma}\label{Lem free of the same rank}
Let $\psi \colon M \to N$ be a homomorphism between free abelian groups of the same finite rank. Then $\psi$ is an isomorphism if and only if $\psi \otimes \F_\ell$ is injective for all primes $\ell$.
\end{Lemma}

\begin{proof}
The cokernel $C$ is nonzero if and only if $C \otimes \F_\ell \neq 0$ for all primes $\ell$, so the result follows from right exactness of $- \otimes \F_\ell$ and a dimension argument.
\end{proof}

\begin{Def}\label{Def r_n}
Write $r_n$ for the number
\[
r_n := \left\{\begin{array}{ll} \frac{(n+1)^2 + 1}{2}, & n \equiv 0 \pmod 2, \\ \frac{(n+1)^2}{2}, & n \equiv 1 \pmod 2.\end{array}\right.
\]
Then $\H_n$ is free of rank $r_n$, with basis given by
\begin{align*}
&(x^i y^j + x^{n-i}y^{n-j})z^n, & (i,j) &\neq (n-i,n-j),\\
&x^iy^j z^n,  & (i,j) &= \left(\frac{n}{2},\frac{n}{2}\right).
\end{align*}
\end{Def}

\begin{Lemma}\label{Lem dimension of R}
The degree $n$ part of the algebra $R = \Z[A,B,C,D]/(G)$ is free of rank $r_n$.
\end{Lemma}

\begin{proof}
It has a basis given by
\[
\left\{A^iB^jC^k\ \bigg|\ \begin{array}{c}i,j,k \in \Z_{\geq 0},\\i+j+2k = n\end{array}\right\} \cup \left\{ A^iB^jC^kD\ \bigg|\ \begin{array}{c}i,j,k \in \Z_{\geq 0},\\i+j+2k = n-2\end{array}\right\}.
\]
Moreover, a simple induction shows that
\[
r'_n := \# \left\{ (i,j,k) \in \Z_{\geq 0}^3\ \bigg|\ i+j+2k = n \right\} = \left\{\begin{array}{ll} \frac{(n+2)^2}{4}, & n \equiv 0 \pmod 2, \\ \frac{(n+2)^2-1}{4}, & n \equiv 1 \pmod 2.\end{array}\right.
\]
Thus, the result follows since $\rk R_n = r'_n + r'_{n-2} = r_n$.
\end{proof}

\begin{proof}[Proof of \autoref{Thm Hodge ring presentation}.]
Clearly $\phi$ is a homomorphism of graded rings, and one easily verifies that $G \in \ker(\phi)$. Thus, we get an induced map
\[
\psi \colon R \to \H_*
\]
between graded rings that have the same rank in each degree by \autoref{Def r_n} and \autoref{Lem dimension of R}. By \autoref{Lem free of the same rank} it suffices to show that $\psi_\ell = \psi \otimes \F_\ell$ is injective for every prime $\ell$.

Note that $R \otimes \F_\ell = \F_\ell[A,B,C,D]/(G)$ is a $3$-dimensional domain since $G$ is irreducible (for example, its restriction modulo $B$ is Eisenstein at $(C) \subseteq \F_\ell[A,C]$ when viewed as a polynomial in $D$). Thus, to show injectivity of $\psi_\ell$, it is enough to show that the image $\im(\psi_\ell) \subseteq \F_\ell[x,y,z]$ has Krull dimension $3$. But the elements $\psi_\ell(A)$, $\psi_\ell(B)$, $\psi_\ell(C)$ are algebraically independent, for example because the Jacobian
\[
J = \begin{pmatrix}
\tfrac{dA}{dx} & \tfrac{dA}{dy} & \tfrac{dA}{dz}\\[.3em]
\tfrac{dB}{dx} & \tfrac{dB}{dy} & \tfrac{dB}{dz}\\[.3em]
\tfrac{dC}{dx} & \tfrac{dC}{dy} & \tfrac{dC}{dz}
\end{pmatrix} = \begin{pmatrix}
yz   & xz   & 1+xy \\
z    & z    & x+y \\
yz^2 & xz^2 & 2xyz
\end{pmatrix}
\]
is invertible at $(x,y,z) = (0,1,1)$.
\end{proof}

\begin{Cor}\label{Cor generated by P^1, E, P^2, and S}
Let $E$ be an elliptic curve, and let $S$ be any surface for which $h^{1,0}(S)-h^{0,1}(S) = \pm 1$ (for example, the surface constructed in \autoref{Prop example Hodge symmetry}). Then $\H_*$ is generated by $\P^1$, $E$, $\P^2$, and $S$, subject only to a monic quadratic equation in $S$ over $\Z[\P^1,E,\P^2]$.
\end{Cor}

\begin{proof}
In the notation of \autoref{Thm Hodge ring presentation}, we have $A = h(\P^1)$, $B = h(E) - h(\P^1)$, and $C = h(\P^1 \times \P^1) - h(\P^2)$. Finally, $D$ can be obtained from $\pm h(S)$ by adding a suitable linear combination of $A^2$, $AB$, $B^2$, and $C$, since $h^{1,0}(S) - h^{0,1}(S) = \pm 1$. This proves the first statement, and the second follows from \autoref{Thm Hodge ring presentation} since $h(S)$ differs from $\pm D$ by a translation in $\Z[\P^1,E,\P^2]$.
\end{proof}

\begin{Cor}\label{Cor Hodge generated by varieties}
The Hodge ring of varieties over $k$ is generated by smooth projective varieties that are defined over $\F_p$, admit a lift to $\Z_p$, and for which the Hodge--de Rham spectral sequence degenerates.
\end{Cor}

\begin{proof}
Clearly $\P^1$ and $\P^2$ can be defined over $\F_p$ and lifted to $\Z_p$. For $E$ we may choose an elliptic curve over $\F_p$ and lift to $\Z_p$, and for $S$ we may choose the surface of \autoref{Prop example Hodge symmetry}. The degeneration claim is \autoref{Rmk Hodge-de Rham lift}.
\end{proof}
\vskip-\lastskip
We deduce from the above presentation the following theorems.

\begin{Thm}
The only universal congruences between the Hodge numbers $h^{i,j}(X)$ of smooth proper varieties $X$ of dimension $n$ over $k$ are those given by Serre duality.\hfill\qedsymbol
\end{Thm}

\begin{Thm}\label{Thm linear relations Hodge}
The only universal linear relations between the Hodge numbers $h^{i,j}(X)$ of smooth proper varieties $X$ of dimension $n$ over $k$ are those given by Serre duality.\hfill\qedsymbol
\end{Thm}

That is, if $m, n \in \Z_{> 0}$ and $\lambda_{i,j} \in \Z$ for $i, j \in \{0, \ldots, n\}$ are such that
\begin{equation}\label{Eq cong}
\sum \lambda_{i,j} h^{i,j} (X) \equiv 0 \pmod m
\end{equation}
for every $X \in \Var_k$ of dimension $n$, then for all $(i,j) \neq (\tfrac{n}{2},\tfrac{n}{2})$ we have
\begin{equation}\label{Eq cong 2}
\lambda_{i,j} \equiv - \lambda_{n-i,n-j} \pmod m,
\end{equation}
and similarly if we take $\lambda_{i,j} \in \Q$ and consider equality instead of congruence mod $m$ in (\ref{Eq cong}) and (\ref{Eq cong 2}).

\begin{Rmk}
While this paper was in preparation, the results of \cite{KS} on linear relations between Hodge numbers in characteristic $0$ were improved to cover all \emph{polynomial} relations \cite{PS}. In positive characteristic, this will appear in joint work between the first author of \cite{PS} and the present author \cite{vDdBP}.
\end{Rmk}

\section{Birational invariants}\label{Sec birational}
The Hodge numbers $h^0(X,\Omega_X^i)$ are birational invariants of a smooth proper variety $X$, and Chatzistamatiou and R\"ulling proved the same for the numbers $h^i(X,\mathcal O_X)$ \cite[Thm.~1]{ChaRul}. We show that these are the only linear combinations of Hodge numbers that are birational invariants.

\begin{Def}
Let $\I \subseteq \H_*$ be the subgroup generated by differences $X - X'$ for $X, X' \in \Var_k$ birational. Note that $\I$ is a (homogeneous) ideal, for if $X \dashrightarrow X'$ is a birational map, then so is $X \times Y \dashrightarrow X' \times Y$ for any $Y \in \Var_k$. Write $\I_n$ for the degree $n$ part of $\I$.
\end{Def}

\begin{Prop}
Consider the quotient map
\[
\phi \colon \H_* \ra \Z[x,y,z]/(xy).
\]
Then
\begin{enumerate}
\item The degree $n$ part of $\im \phi$ is free of rank $2n$, with basis given by
\[
\left\{y^jz^n\ \bigg|\ 0 \leq j \leq n-1\right\} \cup \left\{ x^iz^n\ \bigg|\ 0 < i \leq n-1\right\} \cup \bigg\{ (x^n + y^n)z^n\bigg\}.
\]
\item The kernel of $\phi$ satisfies $\ker \phi = (C) = \I$.
\end{enumerate}
\end{Prop}

\begin{proof}
The first statement is obvious, by looking at the image of the basis
\begin{align*}
&(x^i y^j + x^{n-i}y^{n-j}) z^n,& & (i,j) \neq \left(\tfrac{n}{2},\tfrac{n}{2}\right),\\
&x^i y^j \cdot z^n,& & (i,j) = \left(\tfrac{n}{2},\tfrac{n}{2}\right)
\end{align*}
of $\H_n$. To prove the second statement, note that $C = xy \cdot z \in \ker \phi$. We have
\[
\H_*/(C) = \Z[A,B,D]/(D^2 - ABD),
\]
so a basis for $\H_*/(C)$ is given by $A^iB^j$ and $A^iB^jD$ for $i,j \in \Z_{\geq 0}$. Thus, the degree $n$ part of $\H_*/(C)$ is free of rank $(n+1) + (n-1) = 2n$. Since $\H_*/(C) \ra \im \phi$ is surjective and the degree $n$ parts of both sides are free of the same rank, we conclude that it is an isomorphism. Thus, $(C) = \ker \phi$.

The Hodge numbers $h^{i,0}(X) = h^0(X,\Omega_X^i)$ and $h^{0,j}(X) = h^j(X,\mathcal O_X)$ are birational invariants (for the latter, see \cite{ChaRul}). Since $\phi$ only remembers the $h^{p,0}$ and the $h^{0,q}$, we get $\I \subseteq \ker \phi$. Finally, note that $C \in \I$ because $C = \Bl{\pt}(\P^2) - \P^2$ (or $\P^1 \times \P^1 - \P^2$). Thus, $(C) \subseteq \I$, which finishes the proof.
\end{proof}
\vskip-\lastskip
We deduce the following theorem, which is the analogue of Theorem 2 of \cite{KS}.

\begin{Thm}
The mod $m$ reduction of an integral linear combination of Hodge numbers is a birational invariant of smooth proper varieties if and only if the linear combination is congruent mod $m$ to a linear combination of the $h^{i,0}$ and the $h^{0,j}$ (and their duals $h^{n-i,n}$ and $h^{n,n-j}$).\hfill\qedsymbol
\end{Thm}

\begin{Thm}\label{Thm linear combinations birational invariant}
A rational linear combination of Hodge numbers is a birational invariant of smooth proper varieties if and only if is is a linear combination of the $h^{i,0}$ and the $h^{0,j}$ (and their duals $h^{n-i,n}$ and $h^{n,n-j}$).\hfill\qedsymbol
\end{Thm}

\section{The de Rham ring}\label{Sec de Rham}
In analogy with the Hodge ring, we define a de Rham ring $\DR_*$ whose elements correspond to formal de Rham polynomials of varieties.

\begin{Def}
Consider $\Z[t,z]$ as a graded ring where $t$ has degree $0$ and $z$ has degree $1$. The \emph{de Rham ring} of varieties over $k$ is the graded subring $\DR_* \subseteq \Z[t,z]$ whose degree $n$ part is
\[
\DR_n = \left\{\left(\ \sum_{i=0}^{2n}h^it^i\right)z^n\ \Bigg|\ \begin{array}{c}h^i = h^{2n-i} \text{ for all } i,\\ h^n \text{ is even if } n \text{ is odd}.\end{array}\right\}.
\]
This differs from the notation of \cite{KS}, where $\DR_*$ denotes the ring where moreover $h^i$ is even for every odd degree $i$.
\end{Def}

\begin{Def}
Write  $\dr \colon \Var_k \to \DR_*$ for the map sending an $n$-dimensional variety $X$ to its \emph{(formal) de Rham polynomial}
\[
\dr(X) = \left(\ \sum_{i=0}^{2n}\hdR^i(X) t^i\right)z^n.
\]
Note that the cup product defines a perfect pairing \cite[Thm.\ VII.2.1.3]{Berthelot}
\[
\HdR^i(X) \times \HdR^{2n-i}(X) \to k,
\]
showing that $\hdR^i(X) = \hdR^{2n-i}(X)$. The pairing on $\HdR^n(X)$ is alternating if $n$ is odd, so in that case $\hdR^n(X)$ is even, showing that $\dr(X)$ lands in $\DR_*$. The K\"unneth formula \cite[Cor.\ V.4.2.3]{Berthelot} shows that $\dr(X \times Y) = \dr(X)\cdot\dr(Y)$ for $X, Y \in \Var_k$.
\end{Def}

\begin{Rmk}\label{Rmk Hodge to de Rham}
There is a natural map $s \colon \H_* \to \DR_*$ given by $x, y \mapsto t$ and $z \mapsto z$. The triangle
\begin{equation*}
\begin{tikzcd}[column sep=.1em]
& \Var_k \ar[start anchor={[xshift=-.1em]}]{ld}[swap]{h}\ar[start anchor={[xshift=.1em]}]{rd}{\dr} & \\
\H_* \ar{rr}[swap]{s} & & \DR_*\!\!\!\!\!
\end{tikzcd}
\end{equation*}
does not commute, because $s(h(X)) = \dr(X)$ if and only if the Hodge--de Rham spectral sequence of $X$ degenerates.

However, by \autoref{Cor Hodge generated by varieties} the Hodge ring $\H_*$ can be generated by varieties for which the Hodge--de Rham spectral sequence degenerates. We can use the presentation of $\H_*$ from \autoref{Thm Hodge ring presentation} to get a compatible presentation of $\DR_*$. In particular, $\DR_*$ is generated by the image of $\dr$; see \autoref{Cor de Rham generated by varieties}.
\end{Rmk}

\begin{Thm}\label{Thm de Rham ring presentation}
Consider the graded ring $\Z[A,B,C,D]$ where $A$ and $B$ have degree $1$, and $C$ and $D$ have degree $2$. Then the map
\[
\psi \colon \Z[A,B,C,D] \to \DR_*
\]
given by
\begin{align*}
\phi(A) &= (1+t^2)z, & \phi(C) &= t^2 \cdot z^2,\\
\phi(B) &= 2t \cdot z,  & \phi(D) &= (t+t^3)z^2
\end{align*}
is a surjection of graded rings with $\psi = s \circ \phi$. The kernel of $\psi$ is given by
\[
J = (A^2C-D^2, AB-2D, B^2-4C, BD-2AC).
\]
\end{Thm}

\begin{proof}
Note that $\DR_n$ is free of rank $n+1$, with basis given by
\begin{align*}
&(t^i+t^{2n-i})z^n & & 0 \leq i < n,\\
&t^nz^n & & n \equiv 0 \pmod{2},\\
&2t^nz^n & & n \equiv 1 \pmod{2}.
\end{align*}
Computing the image of $s$ on the basis of $\H_n$ of \autoref{Def r_n}, we see that $s$ is surjective. Since $\phi$ is surjective by \autoref{Thm Hodge ring presentation}, we conclude that $\psi = s \circ \phi$ is surjective. On the other hand, one checks that $J \subseteq \ker \psi$, so we get a surjection
\[
R = \Z[A,B,C,D]/J \twoheadrightarrow \DR_*.
\]
It suffices to show that the degree $n$ part of $R$ is generated by $n+1$ elements. It is generated by $A^iB^jC^kD^\ell$ for $i+j+2k+2\ell = n$. By the relation $B^2 - 4C$ we may assume $j \leq 1$. The relation $AB - 2D$ then shows that we need only consider the monomials $A^iC^kD^\ell$ and $BC^kD^\ell$. Then $BD - 2AC$ allows us to restrict to $A^iC^kD^\ell$ and $BC^k$. Finally, the relation $A^2C-D^2$ shows that $R_n$ is generated by
\begin{align*}
A^iD^\ell, & & C^kD^\ell\ \ (k > 0), & & AC^kD^\ell\ \ (k > 0), & & BC^k.
\end{align*}
If $n$ is even, there are $\tfrac{n}{2}+1$, $\tfrac{n}{2}$, $0$, and $0$ monomials of these types in $R_n$ respectively, and if $n$ is odd there are $\tfrac{n+1}{2}$, $0$, $\tfrac{n-1}{2}$, and $1$ monomials of these types in $R_n$. In both cases, they add up to $n+1$ elements generating $R_n$.
\end{proof}

\begin{Cor}\label{Cor de Rham generated by varieties}
The de Rham ring of varieties over $k$ is generated by smooth projective varieties that are defined over $\F_p$, admit a lift to $\Z_p$, and for which the Hodge--de Rham spectral sequence degenerates.
\end{Cor}

\begin{proof}
This follows from the same statement in \autoref{Cor Hodge generated by varieties} and the surjectivity of $s \colon \H_* \to \DR_*$.
\end{proof}

\section{The Hodge--de Rham ring}\label{Sec Hodge-de Rham}
Because the triangle in \autoref{Rmk Hodge to de Rham} does not commute, the last thing to compute is the image of the diagonal map
\begin{align*}
\Var_k &\to \H_* \times \DR_*\\
X &\mapsto (h(X),\dr(X)).
\end{align*}
Once again, we will define a subring containing the image, and construct enough varieties that generate this subring.

\begin{Def}
Define the ring homomorphisms
\begin{align*}
\chi \colon \H_* &\to \Z[z] & h^{0,0} \colon \H_* &\to \Z[z] & \chi \colon \DR_* &\to \Z[z] & h^0 \colon \DR_* &\to \Z[z]\\
x,y &\mapsto -1, & x,y &\mapsto 0, & t &\mapsto -1, & t &\mapsto 0.
\end{align*}
Then the \emph{Hodge--de Rham ring of varieties over $k$} is the subring
\[
\HDR_* = \left\{\big(a,b\big) \in \H_* \times \DR_* \ \Bigg|\ \begin{array}{c}h^{0,0}(a) = h^0(b),\\\chi(a) = \chi(b).\end{array}\right\}.
\]
Note that $(h(X),\dr(X)) \in \HDR_n$ if $X$ is a smooth and proper variety over $k$ of dimension $n$. Indeed, the Euler characteristic in any spectral sequence is constant between the pages, so $\chi(h(X)) = \chi(\dr(X))$. Moreover, $h^{0,0}(X)$ and $h^0(X)$ agree since they both equal the number of geometric components of $X$.
\end{Def}

\begin{Lemma}\label{Lem kernel of chi times h^0}
The kernel of the morphism $(\chi,h^0) \colon \DR_* \to \Z[z] \times \Z[z]$ given by $t \mapsto (-1,0)$ is generated by $(t+2t^2+t^3)z^2$ and $(t^2+2t^3+t^4)z^3$.
\end{Lemma}

\begin{proof}
Clearly $I = ((t+2t^2+t^3)z^2,(t^2+2t^3+t^4)z^3)$ is contained in $\ker(\chi,h^0)$, so we get a quotient map
\[
R = \DR_*/I \to \Z[z] \times \Z[z].
\]
The image is generated in degree $n$ by $(1,1)$ if $n = 0$, by $(2z^n,0)$ and $(0,z^n)$ if $n$ is odd, and by $(z^n,0)$ and $(0,z^n)$ if $n > 0$ is even. Thus, it suffices to show that $R_n$ can be generated by $1$ element if $n = 0$ and by $2$ elements if $n > 0$.

Under the presentation $\Z[A,B,C,D]/J \cong \DR_*$ of \autoref{Thm de Rham ring presentation}, the elements $D+2C$ and $AC+BC$ map to the generators of $I$, so
\[
R = \Z[A,B,C,D]/(J+(D+2C,AC+BC)).
\]
Pass to the quotient $\Z[A,B,C,D]/(D+2C) \cong \Z[A,B,C]$, where the image of $J+(D+2C,AC+BC)$ is generated by $B^2-4C$, $AB+4C$, and $AC+BC$. Setting $E = A + B$ (corresponding to the class of an elliptic curve), we finally have to compute
\[
R \cong \Z[E,B,C]/(B^2-4C,EB,EC).
\]
Then $R_n$ is generated by $E^iB^jC^k$ for $i+j+2k=n$. Using $B^2-4C$ we reduce to $j \leq 1$. The relation $EB$ shows that for $j = 1$ we may take $i = 0$. By the relation $EC$ we may assume $i = 0$ if $k > 0$. Thus, $R_n$ is generated by
\begin{align*}
E^i, & & C^k\ \ (k > 0), & & BC^k,
\end{align*}
sitting in degree $i$, $2k$, and $2k + 1$ respectively. We see that $R_n$ is generated by $1$ element if $n = 0$ and by $2$ elements if $n > 0$.
\end{proof}

\begin{Rmk}\label{Rmk generated by a threefold}
The ideal generated by $(t+2t^2+t^3)z^2$ contains $(2t^2+4t^3+2t^4)z^3$ and $(t+2t^2+2t^3+2t^4+t^5)z^3$, so we may replace the second generator of $I$ as in \autoref{Lem kernel of chi times h^0} by any element in $I_3$ for which $h^2$ is odd.
\end{Rmk}

\begin{Thm}\label{Thm Hodge-de Rham ring presentation}
Let $\Z[A,B,C,D]$ as in \autoref{Thm Hodge ring presentation} and \autoref{Thm de Rham ring presentation}. Let $S$ be a surface for which $h^{1,0}(S)+h^{0,1}(S) - \hdR^1(S)$ is odd, and let $T$ be a threefold for which $h^{2,0}(T) + h^{1,1}(T) + h^{0,2}(T) - \hdR^2(T)$ is odd. Define the map
\[
\tau \colon \Z[A,B,C,D][S,T] \to \HDR_* \subseteq \H_* \times \DR_*
\]
on $\Z[A,B,C,D]$ by $(\phi,\psi)$, and by
\begin{align*}
\tau(S) = (h(S),\dr(S)), & & \tau(T) = (h(T),\dr(T)).
\end{align*}
Then $\tau$ is surjective. In particular, $\HDR_*$ is generated by varieties.
\end{Thm}

\begin{Ex}\label{Ex S and T}
For $S$ we may take the surface of \autoref{Prop example Hodge-de Rham degeneration}. For $T$, we may take a sufficiently high degree smooth hypersurface in $S \times S$. Indeed, a K\"unneth computation shows that for $S \times S$ the difference $h^{2,0}+h^{1,1}+h^{0,2}-\hdR^2$ is odd. In characteristic $0$, Nakano vanishing implies weak Lefschetz for algebraic de Rham cohomology. In arbitrary characteristic, replacing Nakano vanishing by Serre vanishing shows that for sufficiently high degree hypersurfaces $T \subseteq S \times S$, the Hodge and de Rham numbers in degree $i < \dim T$ agree with those of $S \times S$.

See for example \cite{vDdBP} for a proof of weak Lefschetz for algebraic de Rham cohomology using Nakano vanishing or Serre vanishing.
\end{Ex}

\begin{proof}[Proof of \autoref{Thm Hodge-de Rham ring presentation}]
By \autoref{Cor Hodge generated by varieties} the Hodge ring $\H_*$ is generated by varieties for which the Hodge--de Rham spectral sequence degenerates. Thus, given $(a,b) \in \HDR_*$, we know that $(a,s(a)) \in \im(\tau)$. We reduce to the case $a = 0$, hence $b \in \ker(\chi,h^0)$, which is the ideal $I$ of \autoref{Lem kernel of chi times h^0}. Again using that $(a,s(a)) \in \im(\tau)$, and by surjectivity of $s$ (\autoref{Thm de Rham ring presentation}), it suffices to let $b$ be one of the generators of $I$.

Since $(h(S),s(h(S)))$ is in the image of $\tau$, so is
\[
S' = \bigg(h(S),s(h(S))\bigg)-\tau(S) = \bigg(0,s(h(S))-\dr(S)\bigg).
\]
Writing $b = s(h(S))-\dr(S)$ as $b = (\sum_i b^it^i)z^2$, we have $b^1 = 1$ by assumption. Since $b \in I$ we get $h^0(b) = 0$, hence $b^0 = 0$, so we get $b^4 = 0$, $b^3 = 1$ by Poincar\'e duality. Finally, $\chi(b) = 0$ gives $b^2 = 2$. We conclude that
\[
S' = (0,(t+2t^2+t^3)z^2).
\]
Replacing $S$ by $T$ in this argument, we get an element of $I_3$ for which $h^2$ is odd. By \autoref{Rmk generated by a threefold} any such class can be taken as the second generator for $I$.
\end{proof}
\vskip-\lastskip
This gives all linear congruences between Hodge numbers and de Rham numbers.

\begin{Thm}
The only universal congruences between the Hodge numbers $h^{i,j}(X)$ and the de Rham numbers $\hdR^i(X)$ of smooth proper varieties $X$ of dimension $n$ over $k$ are the congruences spanned by
\begin{itemize}
\item Serre duality: $h^{i,j}(X) = h^{n-i,n-j}(X)$;
\item Poincar\'e duality: $\hdR^i(X) = \hdR^{2n-i}(X)$;
\item Components: $h^{0,0}(X) = \hdR^0(X)$;
\item Euler characteristic: $\sum_{i,j} (-1)^{i+j}h^{i,j}(X) = \sum_i (-1)^i\hdR^i(X)$;
\item Parity of middle cohomology: $\hdR^n(X) \equiv 0 \pmod 2$ if $n$ is odd.\hfill\qed
\end{itemize}
\end{Thm}

\begin{Thm}\label{Thm linear relations Hodge-de Rham}
The only universal linear relations between the Hodge numbers $h^{i,j}(X)$ and the de Rham numbers $\hdR^i(X)$ of smooth proper varieties $X$ of dimension $n$ over $k$ are the relations spanned by
\begin{itemize}
\item Serre duality: $h^{i,j}(X) = h^{n-i,n-j}(X)$;
\item Poincar\'e duality: $\hdR^i(X) = \hdR^{2n-i}(X)$;
\item Components: $h^{0,0}(X) = \hdR^0(X)$;
\item Euler characteristic: $\sum_{i,j} (-1)^{i+j}h^{i,j}(X) = \sum_i (-1)^i\hdR^i(X)$. \hfill\qed
\end{itemize}
\end{Thm}

\phantomsection
\printbibliography
\end{document}